\documentclass[12pt,reqno]{amsart}
\usepackage{fullpage}
\usepackage{times}
\usepackage{dsfont}
\usepackage{amssymb,amsmath,color}

\usepackage{bbm}
\usepackage{color}

\newtheorem{theorem}{Theorem} 
\newtheorem{lemma}[theorem]{Lemma}

\newtheorem{proposition}[theorem]{Proposition}
\newtheorem{corollary}[theorem]{Corollary}
\theoremstyle{definition}

\usepackage{amsrefs}

\def\setupbib{\catcode`@=\active}
\begingroup\lccode`~=`@
  \lowercase{\endgroup\def~}#1#{\gatherkey{#1}}
\def\gatherkey#1#2{\gatherkeyaux{#1}#2\gatherkeyaux}
\def\gatherkeyaux#1#2,#3\gatherkeyaux{\bib{#2}{#1}{#3}}

\renewcommand{\mod}[1]{{\ifmmode\text{\rm\ (mod~$#1$)}\else\discretionary{}{}{\hbox{ }}\rm(mod~$#1$)\fi}}

\def\a{\mathfrak{a}}

\def\c{\mathfrak{c}}

\def\bk{\boldsymbol{k}}

\def\to{\rightarrow}

\def\1{1\!\!1}

\def\Dif{\mathfrak{D}}   
\def\ringO{\mathcal{O}}     
 
\def\norm{{\rm N}}       
\def\trace{{\rm Tr}} 

\newcommand{\bz}{{\boldsymbol{z}}}

\newcommand{\SL}{{\rm SL}_2}

\thanks{Research of the first author is partially supported by an NSERC Discovery Grant.}

\date{\today}


\begin{document}
\author{Alia Hamieh}

\address[Alia Hamieh]{Department of Mathematics and Statistics \\
        University of Northern British Columbia \\
        Prince George, BC V2N4Z9 \\
        Canada}
\email{alia.hamieh@unbc.ca}

\author{Wissam Raji}
\address[Wissam Raji]{
Department of Mathematics\\
American University of Beirut\\
P.O. Box 11-0236\\
Riad El Solh\\
Beirut 1107 2020\\
Lebanon}

\email{wr07@aub.edu.lb}

\title[Non-vanishing of Derivatives of $L$-functions of Hilbert Modular Forms]{Non-vanishing of Derivatives of $L$-functions of Hilbert Modular Forms in the Critical Strip}

\begin{abstract}
In this paper, we show that, on average, the derivatives of $L$-functions of cuspidal Hilbert modular forms with sufficiently large weight $k$ do not vanish on the line segments $\Im(s)=t_{0}$, $\Re(s)\in(\frac{k-1}{2},\frac{k}{2}-\epsilon)\cup(\frac{k}{2}+\epsilon,\frac{k+1}{2})$. This is analogous to the case of classical modular forms.
 \end{abstract}

\maketitle

\section{Introduction}

In \cite{kohnen}, Kohnen proved that given any real number $t_0$ and $\epsilon>0$, then for $k$ large enough, the average of the normalized L-functions $L^*(f,s)$ with $f$ varying over a basis of Hecke eigenforms of weight $k$ on $SL_2(\mathbb{Z})$ does not vanish on the line segment $$ \Im(s)=t_0, \ \ (k-1)/2< \Re(s)<k/2-\epsilon, k/2+\epsilon<\Re(s)<(k+1)/2.$$  Recently in \cite{KSW}, Kohnen, Sengupta and Weigel extended their method and showed a non-vanishing result for the derivatives of $L$-functions associated to modular forms of integer weight on the full group.  In particular, they show that for $k$ large enough, \begin{equation}
\sum_{j=1}^{d}\frac{1}{\left<f_{k,j}, f_{k,j}\right>}\frac{d^n}{ds^n}[L^*(f_{k,j},s)]
\end{equation}
does not vanish on the line segment $\Im(s)=t_0$, $(k-1)/2< \Re(s)<k/2-\epsilon, k/2+\epsilon<\Re(s)<(k+1)/2$.  

In \cite{HR}, we generalized the result of \cite{kohnen} to the context of cuspidal Hilbert modular forms. In order to describe our work, we introduce the following notation. Let $F$ be a totally real number field of  degree $n$ over $\mathbb{Q}$. Let $\ringO_F$ be the ring of integers of $F$, and assume its narrow class number is equal to 1.  For $\bk\in2\mathbb{N}^n$, we denote by $\mathcal{S}_{\bk}(\SL(\ringO_{F}))$ the space of Hilbert cusp forms of weight $\bk$ for $\SL(\ringO_{F})$.  Our main result in \cite{HR} is the following theorem.
\begin{theorem}\label{thm:old}
Let $\bk=(k,k,\dots,k)\in2\mathbb{N}^n$, and  let $\mathcal{B}_{\bk}(\ringO_{F})$ be a basis of normalized Hecke eigenforms of $\mathcal{S}_{\bk}(\SL(\ringO_{F}))$. Let $t_{0}\in\mathbb{R}$ and $\epsilon>0$. Then there exists a constant $C$ depending only on $t_{0}$, $\epsilon$ and $F$ such that for $k>C$ the average \[\sum_{f\in\mathcal{B}_{\bk}(\ringO_{F})}\frac{\Lambda(f,s)}{\left<f,f\right>}\] is non-vanishing for any $s=\sigma+it_{0}$ with $\sigma\in(\frac{k-1}{2},\frac{k}{2}-\epsilon)\cup(\frac{k}{2}+\epsilon,\frac{k+1}{2})$.
\end{theorem}

Here,  we extend the result in \cite{HR} by showing that the derivatives of $L$-functions of cuspidal Hilbert modular forms with sufficiently large weight $k$ do not vanish on the line segments 
\begin{equation*}
\Im(s)=t_{0}, \ \ \ \  \Re(s)\in(\frac{k-1}{2},\frac{k}{2}-\epsilon)\cup(\frac{k}{2}+\epsilon,\frac{k+1}{2}).
\end{equation*}
More precisely, we prove the following theorem.

\begin{theorem}\label{thm:main}
 Let $\mathcal{B}_{\bk}(\ringO_{F})$ be a basis of normalized Hecke eigenforms of $\mathcal{S}_{\bk}(\SL(\ringO_{F}))$. Let $t_{0}\in\mathbb{R}$, $\epsilon>0$ and $\ell\in\mathbb{N}$. Then there exists a constant $C$ depending only on $t_{0}$, $\epsilon$ and $F$ such that for $k>C$ the average \[\sum_{f\in\mathcal{B}_{\bk}(\ringO_{F})}\frac{1}{\left<f,f\right>}\frac{d^{\ell}}{ds^{\ell}}\left(\Lambda(f,s)\right)\] is non-vanishing for any $s=\sigma+it_{0}$ with $\sigma\in(\frac{k-1}{2},\frac{k}{2}-\epsilon)\cup(\frac{k}{2}+\epsilon,\frac{k+1}{2})$.
\end{theorem}
We obtain the following corollary as a direct consequence

\begin{corollary}
Let $t_{0}$, $\epsilon$, $\ell$ and $C$ be as in Theorem \ref{thm:main}. Then for $k>C$ and any $s=\sigma+it_{0}$ with $\sigma\in(\frac{k-1}{2},\frac{k}{2}-\epsilon)\cup(\frac{k}{2}+\epsilon,\frac{k+1}{2})$, there exists a Hecke eigenform $f\in \mathcal{S}_{\bk}(\SL(\ringO_{F}))$ such that $\frac{d^{\ell}}{ds^{\ell}}\left(\Lambda(f,s)\right)\neq 0$.
\end{corollary}

\section{Setting and Preliminaries}
In this note, we work over a totally real number field $F$ of degree $n$ over $\mathbb{Q}$ with ring of integers $\ringO_{F}$. The group of units in $\ringO_{F}$ is denoted by $\ringO_{F}^{\times}$. For simplicity of exposition, we assume that the narrow class number of $F$ is 1. 

The absolute norm of an ideal $\a\subset\ringO_{F}$ is given by $\norm(\a)= [\ringO_{F}:\a]$. The trace and the norm over $\mathbb{Q}$ of an element $x\in F$ are denoted by $\trace(x)$ and $\norm(x)$, respectively. We denote by $\Dif_{F}$ the different ideal of $F$ and  by $d_{F}$ its discriminant over $\mathbb{Q}$. We have the relation $\Dif_{F}=(d_{F})$ and $\norm(\Dif_{F})=|d_{F}|$.

The real embeddings of $F$ are denoted by $\sigma_{j}:x\mapsto x_{j}:=\sigma_{j}(x)$ for $j=1,\dots,n$. We say $x\in F$ is totally positive and write $x\gg0$ if $x_{j}>0$ for all $j$. Moreover, we use $X^{+}$ to denote the set of all totally positive elements in a subset $X$ of $F$. 

To simplify exposition, we will often make use of the following notation. For $c,d\in F$, $z=(z_{1},\dots,z_{n})\in\mathbb{C}^{n}$ and $s\in\mathbb{C}$, we set \[\norm(cz+d)^{s}=\prod_{j=1}^n(c_jz_j+d_j)^s.\] Moreover, we set \[\trace(cz)=\sum_{j=1}^{n}c_{j}z_{j}.\]

Let us now recall the following results which are crucial for establishing Theorem \ref{thm:main} (see the proof of Lemma \ref{lem:bound} below). 
\begin{lemma}\label{normlemmatrotabas}{$\mathrm{(Trotabas}$ \cite[Lemma~2.1]{trotabas}$\mathrm{)}$}
There exist constants $C_{1}$ and $C_{2}$ depending only on $F$ such that $$\forall \xi\in F,\exists\epsilon\in\ringO_{F}^{\times +},\forall j\in\{1,\dots,n\}: \quad C_{1}|\norm(\xi)|^{1/n}\leq|(\epsilon\xi)_{j}|\leq C_{2}|\norm(\xi)|^{1/n}.$$
\end{lemma}

\begin{lemma}\label{lem:luo}{$\mathrm{(Luo}$ \cite{luo}$\mathrm{)}$}
For $\lambda>0$, we have
\begin{equation}
\sum_{\eta\in\ringO_{F}^{\times+}}\prod_{|\eta_j|<1}|\eta_j|^{\lambda}<\infty.\end{equation}  
\end{lemma}

\section{Fourier Expansion of the Kernel Function for Hilbert Modular Forms}
Let us now define the kernel function for Hilbert modular forms over $F$. Let $z=(z_1,z_2,\dots,z_{n})\in\mathbb{H}^{n}$, $k\in2\mathbb{N}$ and let $s\in\mathbb{C}$ with $1<\Re(s)<{k}-1$. We have 
\begin{equation}\label{eqn:kernel-function}R_{s,k}(z)=\gamma_{k}(s)\sum_{ \begin{pmatrix}
a & b  \\
c & d  \end{pmatrix}\in T\backslash\SL(\mathcal{O}_{F})}\norm(cz+d)^{-k}\frac{\norm(az+b)^{-s}}{\norm(cz+d)^{-s}},\end{equation}
where $T=\left\{\begin{pmatrix}
\epsilon & 0 \\
0 & \epsilon^{-1}  \end{pmatrix}: \epsilon\in\ringO_{F}^{\times +}\right\}$ and $\displaystyle{\gamma_{k}(s)=\frac{\left(i^{s}\Gamma(s)\Gamma(k-s)\right)^n}{[\ringO_{F}^{\times}:\ringO_{F}^{\times+}]}}.$

In \cite{HR}, we computed the Fourier expansion of $R_{s,k}(z)$.
\begin{proposition}\label{lem:lemma2}
The function $R_{s,k}(z)$ has the following Fourier expansion:
\[R_{s,k}(z)=\sum_{\substack{\nu\in\Dif_{F}^{-1}\\\nu\gg0}}r_{s,k}(\nu)\exp\left(2\pi i\trace(\nu\bz)\right),\]
where \begin{align*}r_{s,k}(\nu)&=\frac{(2\pi)^{ns}\Gamma^n(k-s)}{\sqrt{|d_{F}|}}\norm(\nu)^{s-1}+(-1)^{n\frac{k}{2}}\frac{(2\pi )^{n(k-s)}\Gamma^n(s)}{\sqrt{|d_{F}|}}\norm(\nu)^{k-s-1}\\&\hspace{0.2in}+\gamma_{k}(s)\frac{(2\pi i)^{nk}\norm(\nu)^{k-1}}{\Gamma^n(k)\sqrt{|d_{F}|}}\sideset{}{^*}\sum_{\substack{(a,c)\in\ringO_{F}\times\ringO_{F} \\ac\neq0\\\gcd(a,c)=1}}\frac{\norm(c)^{s-k}}{\norm(a)^{s}}\exp\left(2\pi i\trace\left(\nu\frac{d_{0}}{c}\right)\right)\\&\hspace{2.5in}\times\prod_{j=1}^n{}_{1}F_{1}\left(s,k,-\frac{2\pi i\nu_{j}}{a_{j}c_{j}}\right).\end{align*}
\end{proposition}

We also showed in \cite{HR} that $R_{s,k}(z)$ can be expressed as follows: \begin{align}\label{eqn:kernel-expansion}
R_{\overline{s},k}&=(-1)^{n\frac{k}{2}}\pi^n2^{n(2-k)}\Gamma^n(k-1)\sum_{f\in\mathcal{B}_{\bk}(\ringO_{F})}\frac{\Lambda(f,s)}{\left<f,f\right>}f, 
\end{align}
where $\mathcal{B}_{\bk}(\ringO_{F})$ is a basis of normalized Hecke eigenforms of $\mathcal{S}_{\bk}(\SL(\ringO_{F}))$,  and \[\Lambda(f,s)=|d_{F}|^{s}(2\pi)^{-ns}\Gamma^n(s)L(f,s).\] Recall that $\Lambda(f,s)$ satisfies the functional equation (see \cite[page~654]{shimura}) \begin{equation}\label{eqn:functional}\Lambda(k-s)=(-1)^{n\frac{k}{2}}\Lambda(s).\end{equation}

\section{Proof of Theorem \ref{thm:main}}\label{sec:proof}
In this section, we prove the main theorem of this paper following the recent work of Kohnen et al. \cite{KSW}. In view of the functional equation \eqref{eqn:functional}, it suffices to consider the left hand side of the critical strip. Hence, we take $s=\frac{k}{2}-\delta-it_{0}$ where $\epsilon<\delta<\frac12$ and $t_{0}\in\mathbb{R}$.

Taking the first Fourier coefficients on both sides of \eqref{eqn:kernel-expansion} and using Proposition \ref{lem:lemma2}, we get 
\begin{align}\label{eqn:main}&\frac{(2\pi)^{ns}\Gamma^n(k-s)}{\sqrt{|d_{F}|}}+(-1)^{n\frac{k}{2}}\frac{(2\pi)^{n(k-s)}\Gamma^n(s)}{\sqrt{|d_{F}|}}\nonumber\\&\hspace{0.3in}+\frac{(-1)^{n\frac{k}{2}}(2\pi )^{nk}i^{ns}}{\sqrt{|d_{F}|}[\ringO_{F}^{\times}:\ringO_{F}^{\times+}]}\nonumber\\&\hspace{0.5in}\times\sideset{}{^*}\sum_{\substack{(a,c)\in\ringO_{F}\times\ringO_{F} \\ac\neq0\\\gcd(a,c)=1}}\frac{\norm(c)^{s-k}}{\norm(a)^s}\exp\left(2\pi i\trace\left(\frac{d_{0}}{c}\right)\right)\prod_{j=1}^{n}{}_{1}f_{1}\left(s,k,-\frac{2\pi i}{a_{j}c_{j}}\right)\nonumber\\&=(-1)^{n\frac{k}{2}}\pi^n2^{n(2-k)}\Gamma^n(k-1)\sum_{f\in\mathcal{B}_{\bk}(\ringO_{F})}\frac{\Lambda(f,s)}{\left<f,f\right>},\end{align}

where \[{}_{1}f_{1}\left(s,k,-\frac{2\pi i}{a_{j}c_{j}}\right)=\frac{\Gamma(k-s)\Gamma(s)}{\Gamma(k)}{}_{1}F_{1}\left(s,k,-\frac{2\pi i}{a_{j}c_{j}}\right).\]

Taking the $\ell$-th derivative of both sides with respect to $s$ gives

\begin{align}\label{eqn:main}&\frac{1}{\sqrt{|d_F|}}\frac{d^{\ell}}{ds^{\ell}}\left[(2\pi)^{ns}\Gamma^n(k-s)\right]+\frac{(-1)^{n\frac{k}{2}}}{\sqrt{|d_{F}|}}\frac{d^{\ell}}{ds^{\ell}}\left[(2\pi)^{n(k-s)}\Gamma^n(s)\right]\nonumber\\&\hspace{0.3in}+\frac{(-1)^{n\frac{k}{2}}(2\pi )^{nk}}{\sqrt{|d_{F}|}[\ringO_{F}^{\times}:\ringO_{F}^{\times+}]}\nonumber\\&\hspace{0.5in}\times\sideset{}{^*}\sum_{\substack{(a,c)\in\ringO_{F}\times\ringO_{F} \\ac\neq0\\\gcd(a,c)=1}}\frac{d^{\ell}}{ds^{\ell}}\left[\frac{\norm(c)^{s-k}}{\norm(a)^s}\exp\left(\frac{\pi}{2} i ns\right)\exp\left(2\pi i\trace\left(\frac{d_{0}}{c}\right)\right)\prod_{j=1}^{n}{}_{1}f_{1}\left(s,k,-\frac{2\pi i}{a_{j}c_{j}}\right)\right]\nonumber\\&=(-1)^{n\frac{k}{2}}\pi^n2^{n(2-k)}\Gamma^n(k-1)\sum_{f\in\mathcal{B}_{\bk}(\ringO_{F})}\frac{1}{\left<f,f\right>}\frac{d^{\ell}}{ds^{\ell}}\left(\Lambda(f,s)\right),\end{align}

Let us consider first the expression 
\[I_1=\frac{1}{\sqrt{|d_F|}}\frac{d^{\ell}}{ds^{\ell}}\left[(2\pi)^{ns}\Gamma^n(k-s)\right].\] We have 
\begin{align*}
I_1&=\frac{1}{\sqrt{|d_F|}}\sum_{j=0}^{\ell}\binom{\ell}{j}\frac{d^{j}}{ds^{j}}\left[(2\pi)^{ns}\right]\frac{d^{\ell-j}}{ds^{\ell-j}}\left[\Gamma^n(k-s)\right]\\&=\frac{1}{\sqrt{|d_F|}}\sum_{j=0}^{\ell}\binom{\ell}{j}\left(n\log\left(2\pi\right)\right)^{j}(2\pi)^{ns}\sum_{\nu_1+\nu_2+\cdots+\nu_n=\ell-j}\frac{(\ell-j)!}{\nu_1!\nu_2!\cdots\nu_n!}\prod_{1\leq t\leq n}(-1)^{\nu_t}\Gamma^{(\nu_t)}(k-s)\\&=\frac{(2\pi)^{ns}}{\sqrt{|d_F|}}\sum_{j=0}^{\ell}(-1)^{\ell-j}\binom{\ell}{j}\left(n\log\left(2\pi\right)\right)^{j}\sum_{\nu_1+\nu_2+\cdots+\nu_n=\ell-j}\frac{(\ell-j)!}{\nu_1!\nu_2!\cdots\nu_n!}\prod_{1\leq t\leq n}\Gamma^{(\nu_t)}(k-s)\\&=\frac{(2\pi)^{ns}}{\sqrt{|d_F|}}\left(n\log(2\pi)\right)^{\ell}\Gamma^{n}(k-s)\\&\hspace{2em}+ \frac{(2\pi)^{ns}}{\sqrt{|d_F|}}\sum_{j=0}^{\ell-1}(-1)^{\ell-j}\binom{\ell}{j}\left(n\log\left(2\pi\right)\right)^{j}\sum_{\nu_1+\nu_2+\cdots+\nu_n=\ell-j}\frac{(\ell-j)!}{\nu_1!\nu_2!\cdots\nu_n!}\prod_{1\leq t\leq n}\Gamma^{(\nu_t)}(k-s).
\end{align*}

Hence,
\begin{align}\label{eqn:I1}
\frac{\sqrt{|d_F|}I_1}{(2\pi)^{ns}\Gamma^n(k-s)}&=\left(n\log(2\pi)\right)^{\ell}\\&\hspace{2em}+ \sum_{j=0}^{\ell-1}(-1)^{\ell-j}\binom{\ell}{j}\left(n\log\left(2\pi\right)\right)^{j}\sum_{\nu_1+\nu_2+\cdots+\nu_n=\ell-j}\frac{(\ell-j)!}{\nu_1!\nu_2!\cdots\nu_n!}\prod_{1\leq t\leq n}\frac{\Gamma^{(\nu_t)}(k-s)}{\Gamma(k-s)}.
\end{align}
In \cite{KSW}, it is observed that $\frac{\Gamma^{(m)(z)}}{\Gamma(z)}$ can be expressed as a polynomial $P\in\mathbb{Z}[\psi,\psi^{(1)},\dots,\psi^{(m)}]$, where $\psi=\frac{\Gamma'}{\Gamma}$ is the digamma function. It is important to note that the highest power of $\psi$ appearing in the polynomial $P$ is $\psi^m$. 


We also recall the following important estimates of the digamma function:

\[\psi(z)\sim\log(z)-\frac{1}{2z}-\sum_{k=1}^{\infty}\frac{B_{2k}}{2kz^{2k}}\]
\[\psi^{(m)}(z)\sim(-1)^{m-1}\left(\frac{(m-1)!}{z^m}+\frac{m!}{2z^{m+1}}+\sum_{k=1}^{\infty}B_{2k}\frac{(2k+m-1)!}{(2k)!z^{2k+m}}\right),\]
for $z\to\infty$ in $|\arg(z)|<\pi$, where $B_n$ is the $n$th Bernoulli number  (see \cite[6.3.18~\&~6.4.11]{stegun}).

For $s=\frac{k}{2}-\delta+it_0$ with $\epsilon<\delta<\frac12$, see that $\frac{\Gamma^{(\nu_t)}(k-s)}{\Gamma(k-s)}=P_t\left(\log(k-s)\right)+ o(1)$ where $P_t$ is a polynomial with integer coefficients and  degree $\nu_t$. It follows that 
\[\prod_{1\leq t\leq n}\frac{\Gamma^{(\nu_t)}(k-s)}{\Gamma(k-s)}=\prod_{1\leq t\leq n}P_t\left(\log(k-s)\right)+ o(1).\] 
\begin{align*}
\frac{\sqrt{|d_F|}I_1}{(2\pi)^{ns}\Gamma^n(k-s)}&=\left(n\log(2\pi)\right)^{\ell}\\&\hspace{2em}+ \sum_{j=0}^{\ell-1}(-1)^{\ell-j}\binom{\ell}{j}\left(n\log\left(2\pi\right)\right)^{j}\sum_{\nu_1+\nu_2+\cdots+\nu_n=\ell-j}\frac{(\ell-j)!}{\nu_1!\nu_2!\cdots\nu_n!}\prod_{1\leq t\leq n}P_t\left(\log(k-s)\right) +o(1)\\&=\tilde{P}\left(\log(k-s)\right)+o(1),
\end{align*}
where $\tilde{P}$ is a polynomial of degree $\ell$ and with leading coefficient $(-1)^{\ell}\sum_{\nu_1+\cdots+\nu_n=\ell}\frac{\ell!}{\nu_1!\cdots\nu_n!}$.

Next, we deal with the sum
\[I_2=\frac{(-1)^{n\frac{k}{2}}}{\sqrt{|d_{F}|}}\frac{d^{\ell}}{ds^{\ell}}\left[(2\pi)^{n(k-s)}\Gamma^n(s)\right].\]

We have \begin{align*}
I_2&=\frac{(-1)^{n\frac{k}{2}}(2\pi)^{nk}}{\sqrt{|d_F|}}\sum_{j=0}^{\ell}\binom{\ell}{j}\frac{d^{j}}{ds^{j}}\left[(2\pi)^{-ns}\right]\frac{d^{\ell-j}}{ds^{\ell-j}}\left[\Gamma^n(s)\right]\\&=\frac{(-1)^{n\frac{k}{2}}(2\pi)^{n(k-s)}}{\sqrt{|d_F|}}\sum_{j=0}^{\ell}\binom{\ell}{j}(-1)^j\left(n\log\left(2\pi\right)\right)^{j}\sum_{\nu_1+\nu_2+\cdots+\nu_n=\ell-j}\frac{(\ell-j)!}{\nu_1!\nu_2!\cdots\nu_n!}\prod_{1\leq t\leq n}\Gamma^{(\nu_t)}(s).
\end{align*}

It follows that 

\begin{align*}
\frac{\sqrt{|d_F|}I_2}{(2\pi)^{ns}\Gamma^n(k-s)}&=(-1)^{n\frac{k}{2}}(2\pi)^{n(k-2s)}\frac{\Gamma^n(s)}{\Gamma^n(k-s)}\sum_{j=0}^{\ell}\binom{\ell}{j}(-1)^j\left(n\log\left(2\pi\right)\right)^{j}\\&\hspace{1in}\times\sum_{\nu_1+\nu_2+\cdots+\nu_n=\ell-j}\frac{(\ell-j)!}{\nu_1!\nu_2!\cdots\nu_n!}\prod_{1\leq t\leq n}\frac{\Gamma^{(\nu_t)}(s)}{\Gamma(s)}.
\end{align*}
If $s=\frac{k}{2}-\delta+it_0$ with $\epsilon<\delta<\frac12$, we have

\begin{align*}\frac{\sqrt{|d_F|}I_2}{(2\pi)^{ns}\Gamma^n(k-s)}&=(-1)^{n\frac{k}{2}+\ell}(2\pi)^{n(k-2s)}(n\log(2\pi))^{\ell}\frac{\Gamma^n(s)}{\Gamma^n(k-s)}\\&+(-1)^{n\frac{k}{2}}(2\pi)^{n(k-2s)}\frac{\Gamma^n(s)}{\Gamma^n(k-s)}\left(\tilde{Q}\left(\log(k-s)\right)+o(1)\right),\end{align*}
where $\tilde{Q}$ is a polynomial of degree $\ell$ and with leading coefficient $\sum_{\nu_1+\cdots+\nu_n=\ell}\frac{\ell!}{\nu_1!\cdots\nu_n!}$.

Observe that (see \cite[6.1.23~\&~6.1.47]{stegun})
 \[\left|\frac{\Gamma^n(s)}{\Gamma^n(k-s)}\right|=\left|\frac{k}{2}+it_0\right|^{-2n\delta}\cdot\left|1+O\left(\frac{1}{\left|\frac{k}{2}+it_0\right|^n}\right)\right|,\] uniformly in $\epsilon<\delta<\frac12$. It follows that $\frac{\sqrt{|d_F|}I_2}{(2\pi)^{ns}\Gamma^n(k-s)}$ tends to $0$ as $k\to\infty$.

We still have to estimate the following sum:
\begin{align*}I_3&=\frac{(-1)^{n\frac{k}{2}}(2\pi )^{nk}}{\sqrt{|d_{F}|}[\ringO_{F}^{\times}:\ringO_{F}^{\times+}]}\nonumber\\&\hspace{0.5in}\times\sideset{}{^*}\sum_{\substack{(a,c)\in\ringO_{F}\times\ringO_{F} \\ac\neq0\\\gcd(a,c)=1}}\frac{d^{\ell}}{ds^{\ell}}\left[\frac{\norm(c)^{s-k}}{\norm(a)^s}\exp\left(\frac{\pi}{2} i ns\right)\exp\left(2\pi i\trace\left(\frac{d_{0}}{c}\right)\right)\prod_{t=1}^{n}{}_{1}f_{1}\left(s,k,-\frac{2\pi i}{a_{t}c_{t}}\right)\right].\end{align*}

We have
\begin{align*}
I_3&=\frac{(-1)^{n\frac{k}{2}}(2\pi )^{nk}}{\sqrt{|d_{F}|}[\ringO_{F}^{\times}:\ringO_{F}^{\times+}]}\sideset{}{^*}\sum_{\substack{(a,c)\in\ringO_{F}\times\ringO_{F} \\ac\neq0\\\gcd(a,c)=1}}\norm(c)^{-k}\sum_{j=0}^{\ell}\frac{\norm(c)^{s}}{\norm(a)^s}\log\left(\frac{\norm(c)}{\norm(a)}\right)^{j}\\&\hspace{9em}\times\frac{d^{\ell-j}}{ds^{\ell-j}}\left[\exp\left(\frac{\pi}{2} i ns\right)\exp\left(2\pi i\trace\left(\frac{d_{0}}{c}\right)\right)\prod_{t=1}^{n}{}_{1}f_{1}\left(s,k,-\frac{2\pi i}{a_{t}c_{t}}\right)\right].\end{align*}

For $1\leq j\leq \ell-1$, we have
\begin{align*}
&\frac{d^{\ell-j}}{ds^{\ell-j}}\left[\exp\left(\frac{\pi}{2} i ns\right)\exp\left(2\pi i\trace\left(\frac{d_{0}}{c}\right)\right)\prod_{t=1}^{n}{}_{1}f_{1}\left(s,k,-\frac{2\pi i}{a_{t}c_{t}}\right)\right]\\&=\sum_{m=0}^{\ell-j}\binom{\ell-j}{m}\exp\left(2\pi i\trace\left(\frac{d_{0}}{c}\right)\right)\left(\frac{\pi}{2}in\right)^{m}\exp\left(\frac{\pi}{2} i ns\right)\frac{d^{\ell-j-m}}{ds^{\ell-j-m}}\left[\prod_{t=1}^{n}{}_{1}f_{1}\left(s,k,-\frac{2\pi i}{a_{t}c_{t}}\right)\right]\\&=\sum_{m=0}^{\ell-j}\binom{\ell-j}{m}\exp\left(2\pi i\trace\left(\frac{d_{0}}{c}\right)\right)\left(\frac{\pi}{2}in\right)^{m}\exp\left(\frac{\pi}{2} i ns\right)\\&\hspace{5em}\times \sum_{\nu_1+\nu_2+\cdots+\nu_n=\ell-j-m}\frac{(\ell-j-m)!}{\nu_1!\nu_2!\cdots\nu_n!}\prod_{t=1}^{n}(-1)^{\nu_t}\frac{d^{\nu_t}}{ds^{\nu_t}}{}_{1}f_{1}\left(s,k,-\frac{2\pi i}{a_{t}c_{t}}\right).
\end{align*}

Therefore,

\begin{align*}
\frac{\sqrt{|d_F|}I_3}{(2\pi)^{ns}\Gamma^n(k-s)}&=\frac{(-1)^{n\frac{k}{2}}(2\pi )^{n(k-s)}}{\Gamma^{n}(k-s)[\ringO_{F}^{\times}:\ringO_{F}^{\times+}]}\sideset{}{^*}\sum_{\substack{(a,c)\in\ringO_{F}\times\ringO_{F} \\ac\neq0\\\gcd(a,c)=1}}\frac{\norm(c)^{s-k}}{\norm(a)^s}\exp\left(\frac{\pi}{2} i ns\right)\exp\left(2\pi i\trace\left(\frac{d_{0}}{c}\right)\right)\\&\hspace{2em}\times\Big[\log\left(\frac{\norm(c)}{\norm(a)}\right)^{\ell}\prod_{t=1}^{n}{}_{1}f_{1}\left(s,k,-\frac{2\pi i}{a_{t}c_{t}}\right)\\&\hspace{5em}+\sum_{j=0}^{\ell-1}\log\left(\frac{\norm(c)}{\norm(a)}\right)^{j}\sum_{m=0}^{\ell-j}\binom{\ell-j}{m}\left(\frac{\pi}{2}in\right)^{m}\\&\hspace{5em}\times \sum_{\nu_1+\nu_2+\cdots+\nu_n=\ell-j-m}\frac{(\ell-j-m)!}{\nu_1!\nu_2!\cdots\nu_n!}\prod_{t=1}^{n}(-1)^{\nu_t}\frac{d^{\nu_t}}{ds^{\nu_t}}{}_{1}f_{1}\left(s,k,-\frac{2\pi i}{a_{t}c_{t}}\right)\Big].\end{align*}

Let us now consider the sums
\begin{align*}
E_{1,\ell}(s,k)&=\sideset{}{^*}\sum_{\substack{(a,c)\in\ringO_{F}\times\ringO_{F} \\ac\neq0\\\gcd(a,c)=1}}\frac{\norm(c)^{s-k}}{\norm(a)^s}\exp\left(\frac{\pi}{2} i ns\right)\exp\left(2\pi i\trace\left(\frac{d_{0}}{c}\right)\right)\\&\hspace{2em}\times\log\left(\frac{\norm(c)}{\norm(a)}\right)^{\ell}\prod_{t=1}^{n}{}_{1}f_{1}\left(s,k,-\frac{2\pi i}{a_{t}c_{t}}\right)
\end{align*}
and
\begin{align*}
E_{2,\ell}(s,k)&=\sideset{}{^*}\sum_{\substack{(a,c)\in\ringO_{F}\times\ringO_{F} \\ac\neq0\\\gcd(a,c)=1}}\frac{\norm(c)^{s-k}}{\norm(a)^s}\exp\left(2\pi i\trace\left(\frac{d_{0}}{c}\right)\right)\exp\left(\frac{\pi}{2} i ns\right)\sum_{j=0}^{\ell-1}\log\left(\frac{\norm(c)}{\norm(a)}\right)^{j}\\&\hspace{2em}\times\sum_{m=0}^{\ell-j}\binom{\ell-j}{m}\left(\frac{\pi}{2}in\right)^{m} \sum_{\nu_1+\nu_2+\cdots+\nu_n=\ell-j-m}\frac{(\ell-j-m)!}{\nu_1!\nu_2!\cdots\nu_n!}\prod_{t=1}^{n}(-1)^{\nu_t}\frac{d^{\nu_t}}{ds^{\nu_t}}{}_{1}f_{1}\left(s,k,-\frac{2\pi i}{a_{t}c_{t}}\right).
\end{align*}
We have
\begin{align*}
E_{1,\ell}(s,k)&=\sum_{\substack{\eta\in\ringO_{F}^{\times+}\\c\in\ringO_{F}/\ringO_{F}^{\times+},\\c\neq0}}\sum_{\substack{a\in\ringO_{F}/\ringO_{F}^{\times+}\\\gcd(a,c)=1}}\frac{\norm(c)^{s-k}}{\norm(a)^s}\exp\left(\frac{\pi}{2} i ns\right)\exp\left(2\pi i\trace\left(\frac{d_{0}}{\eta c}\right)\right)\\&\hspace{2em}\times\log\left(\frac{\norm(c)}{\norm(a)}\right)^{\ell}\prod_{t=1}^{n}{}_{1}f_{1}\left(s,k,-\frac{2\pi i}{a_{t}\eta_t c_{t}}\right),
\end{align*}
and
\begin{align*}
E_{2,\ell}(s,k)&=\sum_{\substack{\eta\in\ringO_{F}^{\times+}\\c\in\ringO_{F}/\ringO_{F}^{\times+},\\c\neq0}}\sum_{\substack{a\in\ringO_{F}/\ringO_{F}^{\times+}\\\gcd(a,c)=1}}\frac{\norm(c)^{s-k}}{\norm(a)^s}\exp\left(2\pi i\trace\left(\frac{d_{0}}{\eta c}\right)\right)\exp\left(\frac{\pi}{2} i ns\right)\\&\hspace{5em}\times\sum_{j=0}^{\ell-1}\log\left(\frac{\norm(c)}{\norm(a)}\right)^{j}\sum_{m=0}^{\ell-j}\binom{\ell-j}{m}\left(\frac{\pi}{2}in\right)^{m}\\&\hspace{7em}\times \sum_{\nu_1+\nu_2+\cdots+\nu_n=\ell-j-m}\frac{(\ell-j-m)!}{\nu_1!\nu_2!\cdots\nu_n!}\prod_{t=1}^{n}(-1)^{\nu_t}\frac{d^{\nu_t}}{ds^{\nu_t}}{}_{1}f_{1}\left(s,k,-\frac{2\pi i}{a_{t}\eta_{t}c_{t}}\right).
\end{align*}
To deal with $E_{1,\ell}(s,k)$ and $E_{2,\ell}(s,k)$, we need the following lemma.

\begin{lemma}\label{lem:hyper-bound}
 Let $s=\frac{k}{2}-\delta+it_{0}$ where $\epsilon<\delta<\frac12$ and $t_{0}\in\mathbb{R}$. Let $\ell$ be a non-negative integer. For all $x\in\mathbb{R}$ and all sufficiently large $k$, we have \[\frac{d^{\ell}}{ds^{\ell}}\left({}_{1}f_{1}\left(s,k,ix\right)\right)\ll_{\ell} \min\left\{1,|x|^{-1}\left(|s-1|+|k-s-1|\right)\right\}.\]
\end{lemma}
\begin{proof}
By \cite[13.2.1]{stegun}, we have \[{}_{1}f_{1}\left(s,k,ix\right)=\int_{0}^{1}\exp(ixu)u^{s-1}(1-u)^{k-s-1}\;du.\]  It follows that $\left|{}_{1}f_{1}\left(s,k,ix\right)\right|\leq 1$ whenever $\Re(s)>1$ and $\Re(k-s)>1$. Moreover, by differentiating both sides with respect to $s$, we get 

\begin{equation}\label{eqn:derivativebound1}\frac{d^{\ell}}{ds^{\ell}}\left({}_{1}f_{1}\left(s,k,ix\right)\right)\ll_{\ell}1\end{equation} for any non-negative integer $\ell$ (see \cite[page 326]{KSW}). On the other hand, integration by parts yields
\begin{align*}
{}_{1}f_{1}\left(s,k,ix\right)&=\int_{0}^{1}\exp(ixu)u^{s-1}(1-u)^{k-s-1}\;du\\&=-\frac{s-1}{ix}\int_{0}^{1}\exp(ixu)u^{s-2}(1-u)^{k-s-1}\;du\\&\hspace{0.5in}+\frac{k-s-1}{ix}\int_{0}^{1}\exp(ixu)u^{s-1}(1-u)^{k-s-2}\;du.
\end{align*} 
Taking the $\ell$-th derivative of both sides with respect to $s$ yields

\begin{align*}
\frac{d^{\ell}}{ds^{\ell}}\left({}_{1}f_{1}\left(s,k,ix\right)\right)
&=-\frac{1}{ix}\int_{0}^{1}\exp(ixu)u^{s-2}(1-u)^{k-s-2}\;du\\&\hspace{0.1in}-\frac{s-1}{ix}\int_{0}^{1}\exp(ixu)\sum_{j=0}^{\ell}(-1)^{\ell-j}\binom{\ell}{j}(\log u)^ju^{s-2}(\log(1-u))^{\ell-j}(1-u)^{k-s-1}\;du\\&\hspace{0.1in}+\frac{k-s-1}{ix}\int_{0}^{1}\exp(ixu)\sum_{j=0}^{\ell}(-1)^{\ell-j}\binom{\ell}{j}(\log u)^ju^{s-1}(\log(1-u))^{\ell-j}(1-u)^{k-s-2}\;du.
\end{align*}

Hence, \begin{equation}\label{eqn:derivativebound2}\frac{d^{\ell}}{ds^{\ell}}\left({}_{1}f_{1}\left(s,k,ix\right)\right)\ll_{\ell} |x|^{-1}\left(|s-1|+|k-s-1|\right)\end{equation} whenever $\Re(s)>2$ and $\Re(k-s)>2$.
The desired result follows from \eqref{eqn:derivativebound1} and \eqref{eqn:derivativebound2}.
\end{proof}


\begin{lemma}\label{lem:bound}
Let $s=\frac{k}{2}-\delta+it_{0}$ where $\epsilon<\delta<\frac12$ and $t_{0}\in\mathbb{R}$. As $k\to\infty$, we have $E_{1,\ell}(s,k)=O(k^n)$ and $E_{2,\ell}(s,k)=O(k^n)$ where the implicit constants depend only on $\ell$, $t_{0}$, $\delta$ and the field $F$.
\end{lemma}
\begin{proof}
Upon taking absolute values, we get 
\begin{align*}
\left|E_{2,\ell}(s,k)\right|&\leq \sum_{\eta\in\ringO_{F}^{\times +}}\sum_{\substack{c\in\ringO_{F}/\ringO_{F}^{\times +},\\c\neq0}}\sum_{\substack{a\in\ringO_{F}/\ringO_{F}^{\times +}\\\gcd(a,c)=1}}|\norm(c)|^{-\frac{k}{2}-\delta}|\norm(a)|^{\delta-\frac{k}{2}}\sum_{j=0}^{\ell-1}\left|\log\left(\frac{\norm(c)}{\norm(a)}\right)\right|^{j}\\&\hspace{5em}\times\sum_{m=0}^{\ell-j}\binom{\ell-j}{m}\left(\frac{\pi}{2}n\right)^{m}\exp\left(\frac{\pi}{2}  nt_0\right)\\&\hspace{5em}\times \sum_{\nu_1+\nu_2+\cdots+\nu_n=\ell-j-m}\frac{(\ell-j-m)!}{\nu_1!\nu_2!\cdots\nu_n!}\prod_{t=1}^{n}\left|\frac{d^{\nu_t}}{ds^{\nu_t}}{}_{1}f_{1}\left(s,k,-\frac{2\pi i}{a_{t}\eta_{t}c_{t}}\right)\right|.
\end{align*}
By Lemma \ref{lem:hyper-bound}, we know that  \[\frac{d^{\nu_t}}{ds^{\nu_t}}\left({}_{1}f_{1}\left(s,k,ix\right)\right)\ll_{\nu_t} \left(|s-1|+|k-s-1|\right) \left|\frac{2\pi}{\eta_{j}a_{j}c_{j}}\right|^{-\omega_j},\] where $\omega_j$ is either $0$ or $1$  depending on whether $|\eta_{j}|\geq1$ or $|\eta_{j}|<1$ respectively. Hence, we get \[\prod_{j=1}^n\left|\frac{d^{\nu_t}}{ds^{\nu_t}}\left({}_{1}f_{1}\left(s,k,ix\right)\right)\right|\ll_{\nu_t} \left(|s-1|+|k-s-1|\right)^{n}\prod_{|\eta_{j}|<1}\left|\frac{2\pi}{\eta_{j}a_{j}c_{j}}\right|^{-1}.\]
It follows that $E_{2,\ell}(s,k)$ is  \begin{equation}\label{eqn:S}\ll\left(|s-1|+|k-s-1|\right)^{n}\sum_{\eta\in\ringO_{F}^{\times +}}\sum_{\substack{c\in\ringO_{F}/\ringO_{F}^{\times +},\\c\neq0}}\sum_{\substack{a\in\ringO_{F}/\ringO_{F}^{\times +}\\\gcd(a,c)=1}}\frac{ \prod_{|\eta_{j}|<1}\left|\frac{2\pi}{\eta_{j}a_{j}c_{j}}\right|^{-1}}{|\norm(c)|^{\frac{k}{2}+\delta-\epsilon}|\norm(a)|^{\frac{k}{2}-\delta+\epsilon}}.\end{equation}
By Lemma~\ref{normlemmatrotabas}, we may assume that the representatives $a,c\in \ringO_{F}/\ringO_{F}^{\times +}$ in (\ref{eqn:S}) satisfy \begin{equation*}\label{eqn:norms}
\norm(a)^{1/n}\ll a_{j}\ll\norm(a)^{1/n}\quad \text{and} \quad |\norm(c)|^{1/n}\ll c_{j} \ll|\norm(c)|^{1/n},\end{equation*}
 for all  $j\in\{1,\cdots,n\}$ with implicit constants depending only on $F$. Therefore, we have \begin{equation*}E_{2,\ell}(s,k)\ll k^n\sum_{\eta\in\ringO_{F}^{\times +}}\prod_{|\eta_{j}|<1}|\eta_{j}|\sum_{\substack{c\in\ringO_{F}/\ringO_{F}^{\times +},\\c\neq0}}\sum_{\substack{a\in\ringO_{F}/\ringO_{F}^{\times +}\\\gcd(a,c)=1}}|\norm(c)|^{-\frac{k}{2}+1-\delta+\epsilon}|\norm(a)|^{-\frac{k}{2}+1+\delta-\epsilon}. \end{equation*}
Lemma \ref{lem:luo} allows us to factor out the sum over all $\eta\in\ringO_{F}^{\times +}$ since it is convergent and depends only on $F$. Thus, we get $E_{2,\ell}(s,k)=O(k^n)$ for sufficiently large $k$ as desired. We also emphasize that the implied constant in this estimate depends only on  $\delta$, $t_{0}$, $\ell$ and the field $F$. The sum $E_{1,\ell}(s,k)$ is treated similarly, and so we will not include the details here. 
\end{proof}
It follows from Lemma \ref{lem:bound} that 
\[\frac{\sqrt{|d_F|}I_3}{(2\pi)^{ns}\Gamma^n(k-s)}\ll\frac{k^n(2\pi)^{n\frac{k}{2}}}{\left|\Gamma^n(\frac{k}{2}+\delta-it_0)\right|},\] which tends to $0$
 as $k\to\infty$. Moreover, we have already established that as $k\to\infty$ we have
\[\frac{\sqrt{|d_F|}I_2}{(2\pi)^{ns}\Gamma^n(k-s)}\rightarrow0,\]  where as \[\frac{\sqrt{|d_F|}I_1}{(2\pi)^{ns}\Gamma^n(k-s)}\sim\tilde{P}(\log(\frac{k}{2}+\delta-it_0)),\] for some polynomial $\tilde{P}$ of degree $\ell$. Applying these estimates to \eqref{eqn:main}, we conclude that \[\sum_{f\in\mathcal{B}_{\bk}(\ringO_{F})}\frac{1}{\left<f,f\right>}\frac{d^{\ell}}{ds^{\ell}}\left(\Lambda(f,s)\right)\] is non-vanishing for any $s=\frac{k}{2}-\delta+it_{0}$ with $\epsilon<\delta<\frac12$ and $t_{0}\in\mathbb{R}$.


\end{document}